\documentclass[a4paper,12pt]{amsart}

\usepackage{a4wide}
\usepackage{pgf,tikz}
\usepackage{amssymb}
\usepackage{enumerate}

\usepackage{amsmath}

\newif\ifdetails
\detailstrue
\newcommand{\DETAIL}[1]%
{\ifdetails\par\fbox{\begin{minipage}{0.9\linewidth}\textit{Detail:}
      #1\end{minipage}}\par\fi}
\newcommand{\TODO}[1]%
{\ifdetails\par\fbox{\begin{minipage}{0.9\linewidth}\textbf{TODO:}
      #1\end{minipage}}\par\fi}

\usepackage{makecell}
\usepackage{relsize}

\newtheorem{lemma}{Lemma}
\newtheorem{proposition}[lemma]{Proposition}
\newtheorem{theorem}[lemma]{Theorem}

\theoremstyle{remark}

\newtheorem{remark}{Remark}

\newtheorem{definition}[lemma]{Definition}
\newtheorem{problem}{Problem}
\DeclareMathOperator{\N}{N}

\usepackage{caption}
\usepackage{subcaption}
\usepackage{float} % provides H as float placement specifier
\usepackage[pdftex,a4paper,
citecolor = blue, colorlinks=true,urlcolor=blue]{hyperref}
\usepackage{mathrsfs}
\usetikzlibrary{arrows}

\newcommand{\old}[1]{{}}
\title[Cut vertex and unicyclic graphs with the maximum number of \ldots]{Cut vertex and unicyclic graphs with the maximum number of connected induced subgraphs}

\author{Audace A. V. Dossou-Olory}
\thanks{The work is supported in part by the National Research Foundation of South Africa, grant 118521. Part of the work was done during the author's visit to the African Institute for Mathematical Sciences (AIMS) ZA, Dec. 2019 -- Jan. 2020}

\address{Audace A. V. Dossou-Olory \\ Department of Mathematics and Applied Mathematics \\ University of Johannesburg \\ P.O. Box 524, Auckland Park, Johannesburg 2006 \\ South Africa}
\email{audace@aims.ac.za}


\begin{document}

\begin{abstract}		
Cut vertices are often used as a measure of nodes' importance within a network. They are those nodes whose failure disconnects a graph. Let $\N(G)$ be the number of connected induced subgraphs of a graph $G$. In this work, we investigate the maximum of $\N(G)$ where $G$ is a unicyclic graph with $n$ nodes of which $c$ are cut vertices. For all valid $n,c$, we give a full description of those maximal (that maximise $\N(.)$) unicyclic graphs. It is found that there are generally two maximal unicyclic graphs. For infinitely many values of $n,c$, however, there is a unique maximal unicyclic graph with $n$ nodes and $c$ cut vertices. In particular, the well-known negative correlation between the number of connected induced subgraphs of trees and the Wiener index (sum of distances) fails for unicyclic graphs with $n$ nodes and $c$ cut vertices: for instance, the maximal unicyclic graph with $n=3,4\mod 5$ nodes and $c=n-5>3$ cut vertices is different from the unique graph that was shown by Tan et al.~[{\em The Wiener index of unicyclic graphs given number of pendant vertices or cut vertices}. J. Appl. Math. Comput., 55:1--24, 2017] to minimise the Wiener index. Our main characterisation of maximal unicyclic graphs with respect to the number of connected induced subgraphs also applies to unicyclic graphs with $n$ nodes, $c$ cut vertices and girth at most $g>3$, since it is shown that the girth of every maximal graph with $n$ nodes and $c$ cut vertices cannot exceed $4$.
\end{abstract}

\maketitle

\section{Introduction and main result}\label{Intro:main}
Real-world graphs are extremely large, which poses great challenges for efficiently analy-sing their structural properties. Subgraphs of a graph are important substructures that can reveal valuable information about the underlying graph~\cite{Alokshiya2019}. For instance, subgraphs can be used to identify building blocks and extract the functional properties of complex networks~\cite{milo2002network}. Subgraph enumeration is therefore important for describing large networks, and some algorithms have been invented for efficiently enumerating all connected induced subgraphs of a $n$-vertex graph; see~\cite{Alokshiya2019,Avis2015,Maxwell2014,Uno2015,YanYeh2006}. For general graphs, the fastest known algorithm in this regard has a linear time delay and appeared very recently in~\cite{Alokshiya2019} with an application to certain protein-protein interaction networks. Although we shall not deal with algorithms in this paper, our interest is still related to enumeration: our goal is to know the maximum number of connected induced subgraphs that a given graph can contain. In many applications however, one is only interested in connected subgraphs of graphs that meet certain structural constraints~\cite{Alokshiya2019,bjorklund2012traveling}. Thus, over the set of all $n$-vertex graphs (or unicyclic graphs), those graphs that extremise the number of connected induced subgraphs are characterised in~\cite{Audacegenral2018}. Paper~\cite{AudaceGirth2018} extends the work done in~\cite{Audacegenral2018} by taking into account other structural parameters such that number of cycles, girth, and number of pendant vertices. Specifically, in~\cite{AudaceGirth2018} the author gave a partial characterisation of the $n$-vertex graphs with $d$ cycles, girth $g$ and $k$ pendant vertices that have the maximum number of connected induced subgraphs; for the special case $d=1$ (i.e. unicyclic graphs) the complete structure of those `maximal' unicyclic graphs was provided. Other extremal results with respect to the number of connected induced subgraphs were also obtained for a given number of vertices and any combination of the above mentioned parameters. 

\medskip
A \emph{cut vertex} of a graph $G$ is a vertex whose deletion increases the number of (connected) components of $G$. Cut vertices are often used as a measure of nodes' importance within a network. Their failure disconnects the network and downgrades its performance (e.g. blocks data transmission). In the recent preprint~\cite{AudaceCut2019}, the author determined the maximum number of connected induced subgraphs in a connected graph with order $n$ and $c$ cut vertices, and also described the structure of those graphs attaining the bound. Moreover, he showed, among other things, that the cycle has the minimum number of connected induced subgraphs among all cut vertex-free connected graphs. Our goal is to extend this analysis further to unicyclic graphs. In this paper, we shall study the maximum number of connected induced subgraphs in a unicyclic graph with order $n$ of which $c$ are cut vertices.

Our main motivation for this study stems from \cite[Theorem~11]{AudaceCut2019}, which states that for general connected graphs with order $n>1$ and $c$ cut vertices, the unique graph that realises the maximum number of connected induced subgraphs is essentially the complete graph $K_{n-c}$ of order $n-c$ with one path attached to each of the vertices of $K_{n-c}$ (see \cite{AudaceCut2019} for a more precise description). Since this extremal graph contains several cycles, it would be natural to impose an upper bound on the number of cycles and carry out the same study. The focus of this paper will be on unicyclic graphs; as we shall see, our main result also applies to unicyclic graphs whose girth is at most $g>3$. We mention that the work of Tan et al.~\cite{tan2017wiener} on the Wiener index (sum of distances between all unordered vertex pairs~\cite{Entringer1976,Wiener1947}) also inspired us, as some of the constructive techniques used in~\cite{tan2017wiener} will be adapted to our current setting.

\medskip
Let us first state the main result of this paper. Before getting to the statement, we need to give some definitions. Fix the integers $n>3$ and $0<c<n-2$.
\begin{enumerate}[i)]
\item Let $r$ be the residue of $n-3$ modulo $n-c$ and $q=\lfloor (n-3)/(n-c) \rfloor$. Set $m_j=q+1$ for all $1\leq j \leq r$, and $m_j=q$ for all $r+1\leq j \leq n-c$. Then we define $\Delta_{n,c}$ to be the graph constructed from the triangle $v_0v_1v_2$ by attaching $n-c-2$ pendant paths of respective lengths $m_1,m_2,\ldots,m_{n-c-2}$ at $v_2$, one pendant path of length $m_{n-c-1}$ at $v_1$, and one pendant path of length $m_{n-c}$ at $v_0$.
\item Set $m=\lfloor n/4 \rfloor$ and let $r$ be the residue of $n$ modulo $4$. Then the graph $\Omega_{n,n-4}$ with order $n$ and $n-4$ cut vertices is constructed from the square $v_0v_1v_2v_3$ by attaching the pendant paths of orders $m_0,m_1,m_2,m_3$ at $v_0,v_1,v_2,v_3$, respectively, where $(m_0,m_1,m_2,m_3)$ is equal to
\begin{align*}
(m,m,m,m),(m+1,m,m,m),(m+1,m+1,m,m),(m+1,m+1,m+1,m)
\end{align*}
when $r$ is equal to $0,1,2,3$, respectively.
\item Let $n>7$ such that $n+k=5m$ for some integer $m$ and some $k\in \{0,1,2\}$. Then the graph $\Omega_{n,n-5}$ with order $n$ and $n-5$ cut vertices is constructed from the square $v_0v_1v_2v_3$ by attaching the pendant paths of orders $m_0,m_1,m_2,m_3$ at $v_0,v_1,v_2,v_3$, respectively, and another pendant of order $m$ at $v_2$, where $(m_0,m_1,m_2,m_3)$ is equal to
\begin{align*}
(m,m,m+1,m),(m,m,m,m),(m-1,m,m,m)
\end{align*}
when $k$ is equal to $0,1,2$, respectively.
\end{enumerate}

{\sc Main result:}
Let the integers $n>3$ and $0< c<n-2$ be given, and $G$ be a unicyclic graph with order $n$ and $c$ cut vertices that maximises the number of connected induced subgraphs. Then the following hold:
\begin{itemize}
\item $G$ is only isomorphic to $\Delta_{n,c}$ if $c=n-3$ or $c<n-5$;
\item $G$ is only isomorphic to $\Omega_{n,c}$ if $c=n-4>1$, or $c=n-5>3$ and $n=3,4\mod 5$;
\item $G$ is isomorphic to both $\Delta_{n,c}$ and $\Omega_{n,c}$ if $c=n-5=3$, or $c=n-4 =1$, or
		\begin{align*}
		\textrm{
			$c=n-5>0$ and $ n=0\mod 5$.}
		\end{align*}
\end{itemize}

\medskip
In contrast to the situation for general graphs where there are relatively few extremal results on this invariant, \emph{number of connected induced subgraphs}, for trees it has been studied in more detail and numerous results are available. For instance, Chung et al.~\cite{Chung1981} estimated the minimum size of that tree $T_n$ with the property that any $n$-vertex tree occurs as subtree of $T_n$; Sz\'ekely and Wang~\cite{szekely2005subtrees,szekely2007binary} determined the structure of both $n$-vertex trees and $n$-leaf binary trees that extremise the number of subtrees; Li and Wang~\cite{LiWang} (resp. Andriantiana et al.~\cite{GreedyWagner}) found the tree with order $n$ and $k$ pendant vertices that minimise (resp. maximise) the number of subtrees; Yan and Yeh~\cite{YanYeh2006} determine the unique tree with maximum degree at least $\Delta$ that has the smallest number of subtrees, and the unique tree with diameter at least $d$ that has the greatest number of subtrees.

\medskip
The content within this paper is structured as follows: We give formal notation and  terminologies in Section~\ref{sec:prelim}. Section~\ref{Sec:Mainpart} contains the technical parts of our approach to the main result. In the process of proving our main theorem, Section~\ref{Sec:Mainpart} is divided into 7 steps in a chronological way. We start by proving in Subsection~\ref{Subsec:Onebranchvertex} that there can only be at most one \emph{branching} vertex in every graph with order $n$ and $c$ cut vertices that has the maximum number of connected induced subgraphs. In Subsection~\ref{Subsec:girth3or4} we show that the girth of every `maximal' graph (that maximises) cannot exceed $4$. We then prove in Subsection~\ref{Subsec:branchvertexcycle} that if there is a branching vertex in a maximal graph $G$, then that vertex must belong to the cycle of $G$. We show in Subsection~\ref{Subsec:almostequalpendantpath} that in a maximal graph whose circle is $C$, every two pendants paths at a branching vertex or adjacent vertices of $C$ must have orders' difference at most $1$. We then proceed to describe in Subsections~\ref{Subsec:girth3structure} and~\ref{Subsec:girth4structure} the full structure of those maximal graphs according to whether the order of $C$ is $3$ or $4$. Finally, Subsection~\ref{Subsec:maintheorem} carries a proof of our main theorem, which characterises those $n$-vertex graphs with $c$ cut vertices that have the greatest number of connected induced subgraphs. Section~\ref{Sec:Conclude} concludes the work and points out a connection of our main result to the Wiener index (sum of distances).

\medskip
All graphs considered in this paper are simple, finite and connected.

\section{Preliminaries}\label{sec:prelim}
The \emph{trivial} graph is that with only one vertex. For a graph $G$ and $u,v \in V(G)$, we shall denote by $\N(G),~\N(G)_v,~\N(G)_{u,v}$ the number of connected induced subgraphs of $G$, those that involve $v$, and those that involve both $u$ and $v$, respectively. On the other hand, $G-v$ (resp. $G-uv$) will denote the graph that results from deleting vertex $v$ (resp. edge $uv$) in $G$, and we simply write $G-u-v$ instead of $(G-u)-v$. More generally, $G-S$ represents the graph obtained from $G$ by deleting all elements of $S$. By a $u-v$ path in $G$ we mean a path that connects vertices $u$ and $v$ in $G$. The number of cut vertices of $G$ will be denoted by $c(G)$ (or simply $c$, when the underlying graph is clear from context). It is a straightforward fact that $c\leq n-2$ holds for every non-trivial graph with order $n$ and $c$ cut vertices, with equality only for the path. As usual, the $n$-vertex path will be denoted by $P_n$. We mention that $\N(P_n)=n(n+1)/2$, and $\N(P_n)_w=n$ if $w$ is an endvertex of $P_n$. 

A vertex of outdegree greater than $1$ in a rooted tree $T$ will be called a \emph{branching vertex} of $T$. If $T$ is a rooted tree and $v \in V(T)$, then we define $T[v]$ to be the \emph{fringe subtree} of $T$ rooted at $v$. In other words, $T[v]$ consists of $v$ and all its descendants in $T$. If $v \in V(G)$, then a \emph{pendant path} at $v$ is a $v-u$ path $P$ with the property that $u$ is a pendant vertex of $G$ and all internal vertices of $P$ have degree $2$ in $G$: we shall refer to $u$ as the \emph{free endvertex} of $P$, even when $P$ is trivial (and $u$ is not a pendant vertex). Moreover, a pendant path of order $2$ will simply be called a \emph{pendant edge}. By $C_n: v_0,v_1,v_2,\ldots, v_{n-1}$ we mean the cycle whose vertices are $v_0,v_1,v_2,\ldots, v_{n-1}$ in this order, i.e. $v_{n-1}$ is adjacent to $v_0$, and $v_j$ is adjacent to $v_{j+1}$ for all $0\leq j \leq n-2$. We shall refer to $C_3$ and $C_4$ as the \emph{triangle} $v_0v_1v_2$ and the \emph{square} $v_0v_1v_2v_3$, respectively. The girth of a graph $G$ is the minimum order among the cycles (if any) of $G$. 

A connected graph that has only one cycle is called a unicyclic graph. Every unicyclic graph $G$ whose girth is $g$ can be constructed from the cycle $C_g$ and some (pairwise vertex disjoint) rooted trees $T_0,\ldots,T_{g-1}$ by identifying the root of $T_j$ with vertex $v_j\in V(C_g)$ for all $0\leq j \leq g-1$; see Figure~\ref{ShapeUnic} for a picture. 
\begin{figure}[!h]\centering
	\resizebox{0.3\textwidth}{!}{%
	\begin{tikzpicture}[line cap=round,line join=round,>=triangle 45,x=1.0cm,y=1.0cm]
	\draw (8.38,8.27) node[anchor=north west] {$v_0$};
	\draw (7.57,6.9) node[anchor=north west] {$v_1$};
	\draw (5.83,6.84) node[anchor=north west] {$v_2$};
	\draw (7.22,9.68) node[anchor=north west] {$v_{g-1}$};
	\draw (9.3,7.86) node[anchor=north west] {$T_0$};
	\draw (8.2,6.1) node[anchor=north west] {$T_1$};
	\draw (4.5,6.52) node[anchor=north west] {$T_2$};
	\draw (8.2,10.6) node[anchor=north west] {$T_{g-1}$};
	\draw (9,8)-- (8,6.36);
	\draw (6.01,6.37)-- (8,6.36);
	\draw (6.01,6.37)-- (5,8);
	\draw (9,8)-- (8.03,9.67);
	\draw (8.03,9.67)-- (5.91,9.66);
	\draw [rotate around={-88.78:(8.01,10.33)},dash pattern=on 1pt off 1pt] (8.01,10.33) ellipse (0.7cm and 0.2cm);
	\draw [rotate around={0.07:(9.64,8)},dash pattern=on 1pt off 1pt] (9.64,8) ellipse (0.66cm and 0.18cm);
	\draw [rotate around={-89.87:(8,5.68)},dash pattern=on 1pt off 1pt] (8,5.68) ellipse (0.72cm and 0.22cm);
	\draw [rotate around={36.22:(5.5,6)},dash pattern=on 1pt off 1pt] (5.5,6) ellipse (0.67cm and 0.24cm);
	
	\fill [color=black] (9,8) circle (2.0pt);
	\fill [color=black] (8,6.36) circle (2.0pt);
	\fill [color=black] (6.01,6.37) circle (2.0pt);
	\fill [color=black] (5,8) circle (0.5pt);
	\fill [color=black] (8.03,9.67) circle (2.0pt);
	\fill [color=black] (5.91,9.66) circle (0.5pt);
	\fill [color=black] (5.25,9.01) circle (0.5pt);
	\fill [color=black] (5.45,9.34) circle (0.5pt);
	\fill [color=black] (5.03,8.73) circle (0.5pt);
	\fill [color=black] (8,11) circle (0.5pt);
	\fill [color=black] (10.28,8) circle (0.5pt);
	\fill [color=black] (8,5) circle (0.5pt);
	\fill [color=black] (4.99,5.62) circle (0.5pt);
	\fill [color=black] (8.14,10.85) circle (0.5pt);
	\fill [color=black] (10.15,7.88) circle (0.5pt);
	\fill [color=black] (7.84,5.17) circle (0.5pt);
	\fill [color=black] (5.16,5.57) circle (0.5pt);
	\end{tikzpicture}}
	\caption{The shape of a unicyclic graph with girth $g$: $T_0,\ldots,T_{g-1}$ are trees rooted at $v_0,v_1,\ldots,v_{g-1}$, respectively.}\label{ShapeUnic}
\end{figure}
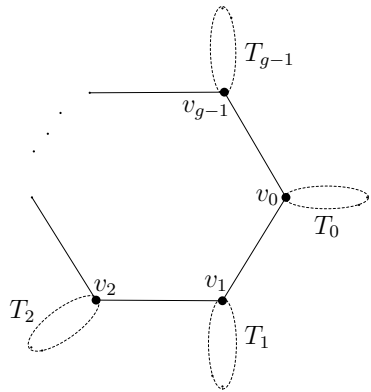

\section{Prescribed number of cut vertices}\label{Sec:Mainpart}
Given two integers $n>2$ and $0\leq c<n-2$, we define $\mathcal{U}(n,c)$ to be the set of all unicyclic graphs with $n$ nodes of which $c$ are cut vertices. Clearly, $\mathcal{U}(n,0)=\{C_n\}$. Throughout the paper, we then assume that $n>3$ and $0< c<n-2$ are fixed integers. By a \emph{maximal graph}, we always mean a graph $G\in \mathcal{U}(n,c)$ that has the maximum number of connected induced subgraphs. The cycle of every maximal graph will be denoted by $C$. 

\medskip
For the purpose of presenting our main result, we need to go through some preparations.

\subsection{At most one branching vertex}\label{Subsec:Onebranchvertex}
Let $G$ be a unicyclic graph whose shape is depicted in Figure~\ref{ShapeUnic}. Then we shall refer to the branching vertices of the rooted trees $T_0,\ldots,T_{g-1}$ as the branching vertices of $G$. The following lemma will be used to show that there can only be at most one branching vertex in every maximal graph.

The first part of Lemma~\ref{singleBranch} below can be found in~\cite{AudaceGirth2018} (the actual formulation in~\cite{AudaceGirth2018} is a bit different but yet equivalent to Lemma~\ref{singleBranch}).

\begin{lemma}\label{singleBranch}
	Let $L,M,R$ be three non-trivial connected graphs whose vertex sets are pairwise disjoints. Let $l \in V(L),~ r\in V(R)$ and $u,v \in V(M)$ be fixed vertices such that $u \neq v$. Denote by $G$ the graph obtained from $L,M,R$ by identifying $l$ with $u$, and $r$ with $v$. Similarly, let $G'$ (resp. $G''$) be the graph obtained from $L,M,R$ by identifying both $l,r$ with $u$ (resp. both $l,r$ with $v$); see Figure~\ref{diagGGpGpp} for a diagram of these graphs.
	\begin{figure}[htbp]\centering
		\definecolor{qqqqff}{rgb}{0.,0.,1.}
		\resizebox{0.7\textwidth}{!}{%
		\begin{tikzpicture}[line cap=round,line join=round,>=triangle 45,x=1.0cm,y=1.0cm]
		rectangle (17.25469285535401,15.543118314192673);
		\draw [dash pattern=on 2pt off 2pt] (8.02,12.78)-- (10.02,12.78);
		\draw [dash pattern=on 2pt off 2pt] (10.02,12.78)-- (10.,12.);
		\draw [dash pattern=on 2pt off 2pt] (10.,12.)-- (10.02,11.26);
		\draw [dash pattern=on 2pt off 2pt] (10.02,11.26)-- (8.,11.26);
		\draw [dash pattern=on 2pt off 2pt] (8.,11.26)-- (8.,12.);
		\draw [dash pattern=on 2pt off 2pt] (8.,12.)-- (8.02,12.78);
		\draw [dash pattern=on 2pt off 2pt] (6.48,12.56)-- (8.,12.);
		\draw [dash pattern=on 2pt off 2pt] (6.48,11.48)-- (8.,12.);
		\draw [dash pattern=on 2pt off 2pt] (6.48,12.56)-- (6.48,11.48);
		\draw [dash pattern=on 2pt off 2pt] (10.,12.)-- (11.48,12.48);
		\draw [dash pattern=on 4pt off 4pt] (10.,12.)-- (11.46,11.46);
		\draw [dash pattern=on 2pt off 2pt] (11.48,12.48)-- (11.46,11.46);
		\draw (8.002875389073733,12.308886980787792) node[anchor=north west] {$l,u$};
		\draw (9.126357554291436,12.19457622267538) node[anchor=north west] {$v,r$};
		\draw (6.736460949518799,12.287683803799489) node[anchor=north west] {$L$};
		\draw (10.70036563986977,12.304751529462257) node[anchor=north west] {$R$};
		\draw [dash pattern=on 2pt off 2pt] (4.,10.)-- (6.,10.);
		\draw [dash pattern=on 2pt off 2pt] (6.,10.)-- (5.98,9.22);
		\draw [dash pattern=on 2pt off 2pt] (5.98,9.22)-- (6.,8.48);
		\draw [dash pattern=on 2pt off 2pt] (6.,8.48)-- (3.98,8.48);
		\draw [dash pattern=on 2pt off 2pt] (3.98,8.48)-- (3.98,9.22);
		\draw [dash pattern=on 2pt off 2pt] (3.98,9.22)-- (4.,10.);
		\draw [dash pattern=on 2pt off 2pt] (3.12,10.28)-- (3.98,9.22);
		\draw [dash pattern=on 2pt off 2pt] (2.6,9.6)-- (3.98,9.22);
		\draw [dash pattern=on 2pt off 2pt] (3.12,10.28)-- (2.6,9.6);
		\draw (4.031552359486242,9.502599906980058) node[anchor=north west] {$l,u,r$};
		\draw (6.006612285785838,9.41691066509247) node[anchor=north west] {$v$};
		\draw (3.008420807842294,10.0050302314765) node[anchor=north west] {$L$};
		\draw (3.087042324067116,8.845683517584332) node[anchor=north west] {$R$};
		\draw [dash pattern=on 2pt off 2pt] (12.,10.)-- (14.,10.);
		\draw [dash pattern=on 2pt off 2pt] (14.,10.)-- (13.98,9.22);
		\draw [dash pattern=on 2pt off 2pt] (13.98,9.22)-- (14.,8.48);
		\draw [dash pattern=on 2pt off 2pt] (14.,8.48)-- (11.98,8.48);
		\draw [dash pattern=on 2pt off 2pt] (11.98,8.48)-- (11.98,9.22);
		\draw [dash pattern=on 2pt off 2pt] (11.98,9.22)-- (12.,10.);
		\draw [dash pattern=on 2pt off 2pt] (13.98,9.22)-- (14.9,10.32);
		\draw [dash pattern=on 4pt off 4pt] (13.98,9.22)-- (15.42,9.66);
		\draw [dash pattern=on 2pt off 2pt] (14.9,10.32)-- (15.42,9.66);
		\draw (12.000266144323993,9.432599906980058) node[anchor=north west] {$u$};
		\draw (12.761465421630708,9.532775213766935) node[anchor=north west] {$v,l,r$};
		\draw (14.542891846445562,10.147962505813732) node[anchor=north west] {$R$};
		\draw [dash pattern=on 2pt off 2pt] (2.66,8.42)-- (3.98,9.22);
		\draw [dash pattern=on 2pt off 2pt] (3.44,7.86)-- (3.98,9.22);
		\draw [dash pattern=on 2pt off 2pt] (2.66,8.42)-- (3.44,7.86);
		\draw (8.654102536581874,12.838560337733874) node[anchor=north west] {$M$};
		\draw (4.70415799076956,10.0550302314765) node[anchor=north west] {$M$};
		\draw (12.718182533719721,10.0550302314765) node[anchor=north west] {$M$};
		\draw [dash pattern=on 2pt off 2pt] (13.98,9.22)-- (15.44,8.52);
		\draw [dash pattern=on 2pt off 2pt] (14.78,7.92)-- (13.98,9.22);
		\draw [dash pattern=on 2pt off 2pt] (15.44,8.52)-- (14.78,7.92);
		\draw [dash pattern=on 2pt off 2pt] (15.44,8.52)-- (14.78,7.92);
		\draw [dash pattern=on 2pt off 2pt] (15.44,8.52)-- (14.78,7.92);
		\draw (14.521338055883508,8.817237308146387) node[anchor=north west] {$L$};
		\draw (8.682548746019819,11.18282315480661) node[anchor=north west] {$G$};
		\draw (4.618293442095093,8.37342849987477) node[anchor=north west] {$G'$};
		\draw (12.675250259382487,8.330496225537535) node[anchor=north west] {$G''$};
		
		\draw [fill=black] (8.02,12.78) circle (0.5pt);
		\draw [fill=black] (10.02,12.78) circle (0.5pt);
		\draw [fill=black] (8.,11.26) circle (0.5pt);
		\draw [fill=black] (10.02,11.26) circle (0.5pt);
		\draw [fill=black] (8.,12.) circle (2.0pt);
		\draw [fill=black] (10.,12.) circle (2.0pt);
		\draw [fill=black] (6.48,12.56) circle (0.5pt);
		\draw [fill=black] (6.48,11.48) circle (0.5pt);
		\draw [fill=qqqqff] (11.48,12.48) circle (0.5pt);
		\draw [fill=qqqqff] (11.46,11.46) circle (0.5pt);
		\draw [fill=black] (4.,10.) circle (0.5pt);
		\draw [fill=black] (6.,10.) circle (0.5pt);
		\draw [fill=black] (3.98,8.48) circle (0.5pt);
		\draw [fill=black] (6.,8.48) circle (0.5pt);
		\draw [fill=black] (3.98,9.22) circle (2.0pt);
		\draw [fill=black] (5.98,9.22) circle (2.0pt);
		\draw [fill=black] (3.12,10.28) circle (0.5pt);
		\draw [fill=black] (2.6,9.6) circle (0.5pt);
		\draw [fill=black] (12.,10.) circle (0.5pt);
		\draw [fill=black] (14.,10.) circle (0.5pt);
		\draw [fill=black] (11.98,8.48) circle (0.5pt);
		\draw [fill=black] (14.,8.48) circle (0.5pt);
		\draw [fill=black] (11.98,9.22) circle (2.0pt);
		\draw [fill=black] (13.98,9.22) circle (2.0pt);
		\draw [fill=black] (15.44,8.52) circle (0.5pt);
		\draw [fill=black] (14.78,7.92) circle (0.5pt);
		\draw [fill=qqqqff] (14.9,10.32) circle (0.5pt);
		\draw [fill=qqqqff] (15.42,9.66) circle (0.5pt);
		\draw [fill=black] (3.44,7.86) circle (0.5pt);
		\draw [fill=black] (2.66,8.42) circle (0.5pt);
		\end{tikzpicture}}
		\caption{The graphs $G,G',G''$ described in Lemma~\ref{singleBranch}.}\label{diagGGpGpp}
	\end{figure}
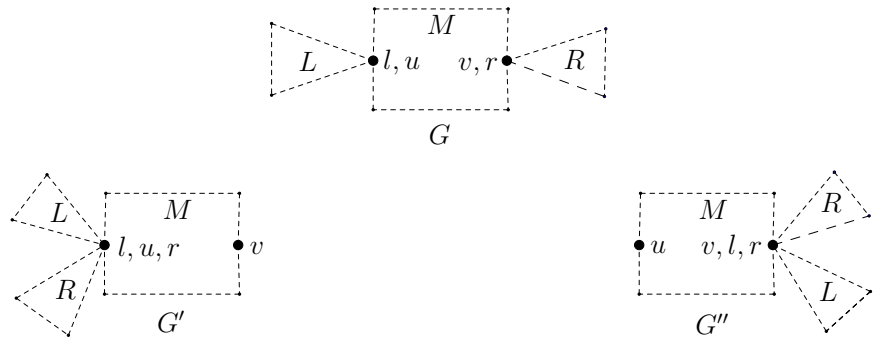

Then it holds that
	\begin{align*}
	\N(G') > \N(G) ~~ \text{or} ~~ \N(G'') > \N(G)\,.
	\end{align*}
Furthermore, if both $u$ and $v$ are cut vertices of the graph $M$, then we have $$c(G')=c(G'')=c(G).$$
\end{lemma}

\begin{proof}
It is shown in the proof of~\cite[Lemma~1]{AudaceGirth2018} that
	\begin{align*}
	\N(G') - \N(G)&= (\N(R)_r -1) (\N(L)_l \cdot \N(M-v)_u - \N(M-u)_v)\,,\\
	\N(G'') - \N(G)&= (\N(L)_l -1) (\N(R)_r \cdot \N(M-u)_v - \N(M-v)_u)\,,
	\end{align*}
Thus we deduce that
\begin{align*}
&\N(G') > \N(G) ~~ \text{if} ~~ \N(M-v)_u \geq  \N(M-u)_v\,,\\ 
&\N(G'') > \N(G) ~~ \text{if} ~~ \N(M-v)_u \leq  \N(M-u)_v\,.
\end{align*}
The setup presented in the lemma shows that both $u$ and $v$ are cut vertices of $G$, and that except possibly $u,v$, all other vertices of $G$ preserve their status (cut vertex/non cut vertex) in both $G'$ and $G''$. Clearly, $u$ (resp. $v$) remains a cut vertex of $G'$ (resp. $G''$). Now by the assumption that both $u$ and $v$ are cut vertices of $M$, we deduce that $u$ (resp. $v$) is a cut vertex of $G''$ (resp. $G'$). Hence $c(G'')=c(G')=c(G)$. 
\end{proof}

As a consequence of Lemma~\ref{singleBranch}, we obtain:
\begin{proposition}\label{Prop:onebranchvertex}
Every maximal graph contains at most one branching vertex.
\end{proposition}

\begin{proof}
Let $G$ be a maximal graph and suppose that $G$ contains two distinct branching vertices, say $u$ and $v$. Let $u'$ be a neighbour of $u$ in $T[u]$ and $v'$ a neighbour of $v$ in $T[v]$. Denote by $L$ (resp. $R$) the component of $T[u]-uu'$ that contains $u$ (resp. the component of $T[v]-vv'$ that contains $v$). Then both $L$ and $R$ are non-trivial graphs. Set $M=G-(V(L-u) \cup V(R-v))$. The specific choice of $u'$ depends on the following two cases:

Case 1: $V(T[u]) \cap V(T[v])\neq \emptyset$. 

We must have $V(T[u]) \subset V(T[v])$ or $V(T[v]) \subset V(T[u])$. We can assume, without loss of generality, that $V(T[v]) \subset V(T[u])$. Choose $u'$ to be that vertex lying on the unique $u-v$ path in $G$ (possibly $u'=v$). Then we have $V(L) \cap V(R)=\emptyset$.

Case 2: $V(T[u]) \cap V(T[v])=\emptyset$.

We have $V(L) \cap V(R)=\emptyset$.

\medskip
Let $w\notin \{u,v\}$ be a vertex of $G$ that lies on its cycle. Note that for either case, we have $w,u',v'\in V(M)$. Moreover, all $u'-w$ paths in $M$ must pass through $u$, and all $v'-w$ paths in $M$ must pass through $v$. Therefore, both $u$ and $v$ are cut vertices of $M$. In particular, Lemma~\ref{singleBranch} applied to $G$ yields a new graph with order $|V(G)|$ and $c(G)$ cut vertices that has more connected induced subgraphs than $G$, a contradiction to the choice of $G$. Hence $G$ can only contains at most one branching vertex.
\end{proof}

\subsection{Girth is at most $4$}\label{Subsec:girth3or4}
By refining an approach employed in~\cite{Audacegenral2018}, we shall prove that the girth of every maximal graph cannot exceed $4$.

\begin{lemma}\label{Girth3or4}
Consider $g>3$ connected graphs $G_0,G_1,\ldots,G_{g-1}$ whose vertex sets are pairwise disjoint. For every $j \in \{0,1,\ldots,g-1\}$, let $v_j$ be a fixed vertex of $G_j$. Assume that $$\N(G_{g-2})_{v_{g-2}}=\max_{0\leq j \leq g-1} \N(G_j)_{v_j}\,,$$ and let the graphs $G$ and $G'$ be constructed as follows:
\begin{itemize}
	\item Add the edges $v_0v_1, v_1v_2,\ldots, v_{g-1}v_0$ to obtain the graph $G$;
	\item Add the edges $v_0v_1, v_1v_2,\ldots, v_{g-2}v_0$ and $v_{g-2}v_{g-1}$ to obtain the graph $G'$.	
\end{itemize}
The following hold:
\begin{enumerate}[i)]
	\item If $G$ is not a cycle, then $c(G')=c(G)$.
	\item For $g>4$ we have $\N(G')>\N(G)$.
	\item For $g=4$ we have $\N(G')\geq \N(G)$ if and only if $\N(G_2)_{v_2}\geq  \N(G_3)_{v_3}(1+\N(G_1)_{v_1})$. Moreover, $\N(G')=\N(G)$ if and only if $\N(G_2)_{v_2}= \N(G_3)_{v_3}(1+\N(G_1)_{v_1})$.
\end{enumerate}
\end{lemma}

\begin{proof}
We first introduce the following notation: $M_1$ is the number of connected induced subgraphs of $G-(V(G_{g-1})\cup V(G_{g-2} - v_{g-2}))$ that contain $v_{g-2}$ and at least one other vertex; $M_2$ is the number of connected induced subgraphs of $G-(V(G_{g-2})\cup V(G_{g-1} - v_{g-1}))$ that contain  $v_{g-1}$ and at least one other vertex; $M_3$ is the number of connected induced subgraphs of $G-(V(G_{g-2} -v_{g-2}) \cup V(G_{g-1} - v_{g-1})) $ that contain both $v_{g-2}$ and $v_{g-1}$; $M_4$ is the number of connected induced subgraphs of $G'-(V(G_{g-1}) \cup V(G_{g-2} -v_{g-2}))$ that contain $v_{g-2}$ and at least one other vertex. Thus it is not difficult to see that
\begin{align*}
M_1=\sum_{r=3}^g \prod_{j=3}^r \N(G_{g-j})_{v_{g-j}}\,, \quad M_2=\sum_{l=0}^{g-3} \prod_{j=0}^l \N(G_j)_{v_j}\,,
\end{align*}
and
\begin{align*}
M_3=\Big(1+ \sum_{r=3}^{g-1} \prod_{j=3}^r \N(G_{g-j})_{v_{g-j}} \Big) + \sum_{l=0}^{g-4} \prod_{j=0}^l \N(G_j)_{v_j} + \sum_{l=0}^{g-4} \sum_{r=3}^{g-1-l} \prod_{j=0}^l \N(G_j)_{v_j} \cdot \prod_{j=3}^r \N(G_{g-j})_{v_{g-j}}\,.
\end{align*}
In the expression of $M_3$, the first term into brackets contributes to those subgraphs that do not involve $v_0$; the second sum contributes to those subgraphs that involve $v_0$ but not $v_{g-3}$; the last sum is the number of those subgraphs that contain both $v_0$ and $v_{g-3}$. In particular, for $g>4$ we get 
\begin{align*}
M_3&=1+M_1+M_2 + \N(G_0)_{v_0} \sum_{r=3}^{g-2} \prod_{j=3}^r \N(G_{g-j})_{v_{g-j}} + \sum_{l=1}^{g-5} \sum_{r=3}^{g-1-l} \prod_{j=0}^l \N(G_j)_{v_j} \cdot \prod_{j=3}^r \N(G_{g-j})_{v_{g-j}}\\
& > 1+M_1+M_2\,.
\end{align*}
With the above notation, we can infer that the number of connected induced subgraphs of $G$ that contain
\begin{enumerate}[(i)]
	\item a vertex of $G_{g-2}$ and no vertex of $G_{g-1}$ is
	\begin{align*}
	\N(G_{g-2})+ \N(G_{g-2})_{v_{g-2}} M_1\,,
	\end{align*}
	\item a vertex of $G_{g-1}$ and no vertex of $G_{g-2}$ is
	\begin{align*}
	\N(G_{g-1})+ \N(G_{g-1})_{v_{g-1}} M_2\,,
	\end{align*}
	\item a vertex of $G_{g-2}$ and a vertex of $G_{g-1}$ is
	\begin{align*}
	\N(G_{g-2})_{v_{g-2}} \N(G_{g-1})_{v_{g-1}} M_3\,.
	\end{align*}
\end{enumerate}
Let us denote by $\N(G;V(G_{g-2})\cup V(G_{g-1}))$ the number of connected induced subgraphs of $G$ that contain an element of the set $V(G_{g-2})\cup V(G_{g-1})$. Thus we have
\begin{align}\label{NumberG}
\begin{split}
\N(G;V(G_{g-2})\cup V(G_{g-1}))=&\N(G_{g-2})+ \N(G_{g-2})_{v_{g-2}} M_1 + 	\N(G_{g-1})+ \N(G_{g-1})_{v_{g-1}} M_2 \\
&+ \N(G_{g-2})_{v_{g-2}} \N(G_{g-1})_{v_{g-1}} M_3\,.
\end{split}
\end{align}
Denote by $G'_{g-2}$ the subgraph induced by $V(G_{g-2}) \cup V(G_{g-1})$ in $G'$. Thus we have
\begin{align*}
\N(G'_{g-2})_{v_{g-2}}&=\N(G_{g-2})_{v_{g-2}}(1+\N(G_{g-1})_{v_{g-1}})\,,\\
\N(G'_{g-2}- v_{g-2})&=\N(G_{g-2}- v_{g-2}) +\N(G_{g-1})\,,
\end{align*}
and the number $\N(G';V(G_{g-2})\cup V(G_{g-1}))$ of connected induced subgraphs of $G'$ that contain an element of $V(G_{g-2}) \cup V(G_{g-1})$ is given by
\begin{align}\label{NumberGprime}
\begin{split}
\N(G';V(G_{g-2})\cup V(G_{g-1}))=&\N(G'_{g-2}) + \N(G'_{g-2})_{v_{g-2}} M_4\\
=&\N(G_{g-2})_{v_{g-2}}(1+\N(G_{g-1})_{v_{g-1}})(1+M_4)\\
& + \N(G_{g-2}- v_{g-2}) +\N(G_{g-1})\,.
\end{split}
\end{align}
By construction $G-(V(G_{g-2})\cup V(G_{g-1}))$ and $G'-(V(G_{g-2})\cup V(G_{g-1}))$ are isomorphic graphs, and it is clear that $M_3=1+M_4$. Hence the difference~\eqref{NumberGprime} -~\eqref{NumberG} gives
\begin{align*}
\N(G';V(G_{g-2})&\cup V(G_{g-1})) - \N(G;V(G_{g-2})\cup V(G_{g-1})) \\
&=\N(G_{g-2})_{v_{g-2}} (M_3 -1) - \N(G_{g-2})_{v_{g-2}} M_1 - \N(G_{g-1})_{v_{g-1}} M_2\\
&> M_2 \big(\N(G_{g-2})_{v_{g-2}} - \N(G_{g-1})_{v_{g-1}}\big)\,,
\end{align*}
where the last step uses the inequality $M_3> 1+ M_1+M_2$ for $g>4$. It follows from the assumption $\N(G_{g-2})_{v_{g-2}} \geq  \N(G_{g-1})_{v_{g-1}}$ that 
\begin{align*}
\N(G')-\N(G) = \N(G';V(G_{g-2})\cup V(G_{g-1})) - \N(G;V(G_{g-2})\cup V(G_{g-1}))>0
\end{align*}
for $g>4$. This completes the proof of ii). For $g=4$, the expressions of $M_1,M_2,M_3$ become
\begin{align*}
M_1&=\N(G_1)_{v_1}(1+\N(G_0)_{v_0}), ~M_2=\N(G_0)_{v_0}(1+\N(G_1)_{v_1})\,,\\ M_3&=(1+\N(G_0)_{v_0})(1+\N(G_1)_{v_1})\,,
\end{align*}
and they imply that
\begin{align*}
\N(G')-\N(G)&=\N(G';V(G_2)\cup V(G_3)) - \N(G;V(G_2)\cup V(G_3))\\
&=\N(G_0)_{v_0}(\N(G_2)_{v_2}- \N(G_3)_{v_3}(1+\N(G_1)_{v_1}))\,.
\end{align*}
This proves iii).

It remains to prove i). It is clear that all cut vertices (resp. non cut vertices), except possibly vertices $v_{g-2}$ and $v_{g-1}$ of $G$ remain cut vertices (resp. non cut vertices) of $G'$. On the other hand, $v_{g-1}$ is not a cut vertex of $G$ if and only if $|V(G_{g-1})|=1$, in which case $v_{g-1}$ becomes a pendant vertex of $G'$. Since $G$ is not a cycle, we have $|V(G_{g-2})|>1$ which implies that $v_{g-2}$ is a cut vertex of $G$ and thus a cut vertex of $G'$. This proves that $c(G')=c(G)$, completing the proof of the lemma.
\end{proof}

From now on, we assume that $G$ is described as in Lemma~\ref{Girth3or4}. The following proposition is an immediate consequence of Lemma~\ref{Girth3or4}:

\begin{proposition}\label{Prop:Girth3or4}
Every maximal graph has girth $3$ or $4$. If a maximal graph $G$ has girth $4$, then it holds that $$\N(G_2)_{v_2}\leq  \N(G_3)_{v_3}(1+\N(G_1)_{v_1}).$$
\end{proposition}

\medskip
The final lemma of this subsection will be needed in subsequent subsections. 

\begin{lemma}\label{Formula3or4}
	Let $G$ be the graph described in Lemma~\ref{Girth3or4}. The following hold:
	
	For $g=3$, we have
	\begin{align*}
	\N(G)=&\N(G_0-v_0) + \N(G_1-v_1) +\N(G_2-v_2)  \\
	& + (1+\N(G_0)_{v_0}) (1+\N(G_1)_{v_1}) (1+\N(G_2)_{v_2}) -1\,.
	\end{align*}
	
	For $g=4$, we have
	\begin{align*}
	\N(G)=&\N(G_0-v_0) + \N(G_1-v_1) +\N(G_2-v_2) +\N(G_3-v_3) \\
	& + (1+\N(G_0)_{v_0}) (1+\N(G_1)_{v_1}) (1+\N(G_2)_{v_2}) (1+\N(G_3)_{v_3})\\ 
	& - (1+\N(G_0)_{v_0}\cdot \N(G_2)_{v_2} + \N(G_1)_{v_1}\cdot \N(G_3)_{v_3})\,.
	\end{align*}
\end{lemma}

\begin{proof}
Let $C$ be the cycle of $G$. Simply categorise the connected induced subgraphs of $G$ according to whether they contain a vertex of $C$ or not.
\end{proof}

\subsection{The branching vertex lies on the cycle}\label{Subsec:branchvertexcycle}
Our next goal is to prove that if there is a branching vertex in a maximal graph $G$, then it must belong to the cycle of $G$.

We begin with other graph transformations that preserve the number of cut vertices of $G$ while affecting the number of connected induced subgraphs. The main result of this subsection will follow after a series of lemmas.

\begin{lemma}\label{G1G2merged}
	If $G$ is obtained from two vertex disjoint connected graphs $G_1,G_2$ by fixing $u_1 \in V(G_1),u_2\in V(G_2)$ and identifying $u_1$ with $u_2$, then it holds that
	\begin{align*}
	\N(G)=\N(G_1)+\N(G_2) -1 + (\N(G_1)_{u_1} -1) (\N(G_2)_{u_2} -1)\,.
	\end{align*}
\end{lemma}

\medskip
The proof of Lemma~\ref{G1G2merged} is straightforward by noting that $\N(G_1)+\N(G_2) -1$ counts those connected induced subgraphs of $G$ that are contained entirely in $G_1$ or $G_2$. Lemma~\ref{G1G2merged} is used in the next two lemmas without further reference.

\begin{lemma}\label{AB}
Let $u \in V(A), v\in V(B)$ be fixed vertices of two vertex disjoint connected graphs $A$ and $B$ such that $|V(A)|>1$. Attach a pendant path of order $n_1$ at $u$, and connect $u$ and $v$ by another path of order $n_2$ for some $n_1,n_2>1$ to obtain the graph $G$. If $G'$ (resp. $G''$) is obtained the same way by replacing $n_1$ with $n_1-1$ and $n_2$ with $n_2+1$ (resp. $n_1$ with $n_1+1$ and $n_2$ with $n_2-1$), then we have $$\N(G')>\N(G)$$ if and only if $\N(B)_v<n_1-n_2$, and $$\N(G'')>\N(G)$$ if and only if $\N(B)_v>n_1-n_2+2$. Furthermore, we have 
	\begin{align*}
	c(G')=c(G) ~~ \text{if} ~~ n_1> 2\,, ~~ \text{and} ~~ c(G'')=c(G) ~~ \text{if} ~~ n_2>2 ~\text{or} ~ |V(B)|>1\,.
	\end{align*}
\end{lemma}

\begin{proof}	
Simply note that
	\begin{align*}
	\N(G)=&\N(G)_{u,v} +\N(G-v)_u +\N(G-u)_v+ \N(G-u-v)\\
	=&~n_1\cdot \N(A)_u \cdot \N(B)_v + n_1(n_2-1)\N(A)_u + (n_2-1)\N(B)_v\\
	& + \N(A-u) + \N(B-v) + \binom{n_1}{2} +  \binom{n_2 -1}{2}\,.
	\end{align*}
Expressions for $\N(G')$ and $\N(G'')$ can be obtained in a similar way. In particular, we get
	\begin{align*}
	\N(G')-\N(G)&=(\N(A)_u -1)(n_1-n_2 - \N(B)_v)\,,\\
	\N(G'')-\N(G)&=-(\N(A)_u -1)(n_1-n_2 +2 - \N(B)_v)
	\end{align*}
after simplification, where in the latter identity it is assumed that $n_2>2$. For $n_2=2$, vertices $u$ and $v$ coincide in $G''$. This gives us
\begin{align*}
\N(G'')_u =(n_1+1) \N(A)_u \cdot \N(B)_v\,, \quad \N(G''-u)= \N(A-u) + \N(B-v) + \binom{n_1+1}{2}\,.
\end{align*}
Thus we get
	$$\N(G'')-\N(G)=-(\N(A)_u -1)(n_1 - \N(B)_v)\,,$$ which completes the proof of the first part.
	
	\medskip
	If $n_1>2$ then the graph $G'$ can be obtained from $G$ by first contracting one vertex, say $x \neq u$ of the pendant path at $u$, and then inserting $x$ into the $u-v$ path (i.e. subdividing one edge). This construction shows that all vertices of $G$ preserve their status (cut vertex or not) in $G'$. Likewise, if $n_2>2$ then the graph $G''$ can be obtained from $G$ by first contracting one vertex, say $x \notin \{u,v\}$ of the $u-v$ path and then inserting $x$ into the pendant path at $u$. Thus we have $c(G'')=c(G)$. On the other hand, if $n_2=2$ and $|V(B)|>1$, then $G''$ can be constructed from $G$ by shrinking the edge $uv$ (i.e. identifying $u$ with $v$) and inserting a new vertex, say $y$ into the pendant path at $u$. Since $y$ is a cut vertex of $G''$ and $v$ is a cut vertex of $G$, we obtain $c(G'')=c(G)$.
\end{proof}

Note that if the graph $B$ in Lemma~\ref{AB} is trivial, then we can assume, without loss of generality that $n_1\geq n_2$. Therefore, we obtain the following important remark:

\begin{remark}\label{Rem:pendantSingleVertex}
Let $G, G',G''$ be the graphs defined in Lemma~\ref{AB}, and assume that $|V(B)|=1$. Then we have $$\N(G)\geq \N(G') \quad \text{and} \quad \N(G)\geq \N(G'')$$ if and only if $|n_1-n_2|\leq 1$. Moreover, if $n_2>2$ then $c(G'')=c(G')=c(G)$.
\end{remark}

\begin{lemma}\label{DeleteAddLeaf}
	Let $H$ be a connected graph and $x,y$ be two distinct vertices of $H$. Let $G$ (resp. $G'$) be obtained from $H$ by attaching a pendant edge at $y$ (resp. at $x$). Then it holds that
	\begin{align*}
	\N(G')-\N(G)=\N(H)_x - \N(H)_y=\N(H-y)_x - \N(H-x)_y\,.
	\end{align*}
	Furthermore, we have $$c(G')=c(G)$$ if and only if $x$ and $y$ have the same status (cut vertex or not) in $H$.
\end{lemma}

\begin{proof}
	For the first part of the lemma, we note that
	\begin{align*}
	\N(G')-\N(G)&=\N(H)_x -\N(H)_y =(\N(H)_{x,y} + \N(H-y)_x) - (\N(H)_{y,x} + \N(H-x)_y)\\
	& = \N(H-y)_x - \N(H-x)_y\,.
	\end{align*}
	All vertices except possibly $x,y$ of $G$ preserve their status (cut vertex or not) in $G'$. By construction, $x$ is a cut vertex of $G'$ and $y$ is a cut vertex of $G$. Thus $c(G')=c(G)$ holds if and only if the status of $x$ in $G$ (thus in $H$) is the same as the status of $y$ in $G'$ (thus in $H$).
\end{proof}

We can now state and prove the main result of this subsection:

\begin{proposition}\label{branchvertcycle} 
The branching vertex in a maximal graph $G$ must lie on the cycle of $G$.
\end{proposition}

\begin{proof}
Let $C$ be the cycle of $G$ and $w$ the branching vertex of $G$. Among all vertices of $C$, say without loss of generality that $v_0$ is at minimum distance from $w$ in $G$. Suppose that $w\neq v_0$. Denote by $z$ the neighbour of $w$ (possible $z=v_0$) that lies on the $w-v_0$ path in $G$, and let $P_m$ be a fixed pendant path at $w$. Then $m>1$ since $w$ is a branching vertex of $G$. Denote by $B$ the component of $G-wz$ that contains $z$. We observe two cases:

Case 1: $\N(B)_z>m$.

Let $y'$ be the neighbour of $w$ that lies on $P_m$. Denote by $A$ the component of $G-\{wz,wy'\}$ that contains $w$. Then $|V(A)|>1$ since $w$ is a branching vertex of $G$. Set $u=w$ and $v=z$. From this decomposition, we can apply Lemma~\ref{AB} with $n_1=m$ and $n_2=2$; since $|V(B)|>1$ and by assumption $\N(B)_v > m=n_1=n_1-n_2+2$, the graph $G$ can be transformed into another graph $G''$ that satisfies $c(G'')=c(G)$ and $\N(G'')>\N(G)$. This contradicts the maximality of $G$. Note that the transformation that takes $G$ to $G''$ in Lemma~\ref{AB} decreases the length of the $u-v$ path by exactly $1$.

Case 2: $\N(B)_z \leq m$. 

Let $n_1$ (resp. $n_2$) be the order of the pendant path $G_1$ at $v_1$ (resp. $G_2$ at $v_2$), and $t$ the order of the $w-v_0$ path in $G$. Denote by $x$ the free endvertex (possibly $x=v_1$) of $G_1$, by $u$ the free endvertex of $P_m$, and by $y$ the neighbour of $u$ (possibly $y=w$) in $G$. Note that $G-V(B)$ is a star-like tree rooted at $w$. Set $H=G-u$ and $A=G-(V(B) \cup V(P_m-w))$. Then both $x$ and $y$ are non cut vertices of $H$. This is because $y$ is a pendant vertex of $H$ as $m \geq \N(B)_z>2$. Construct a new graph $G'$ from $G$ by deleting the edge $uy$ and adding the edge $ux$. Note that this transformation that takes $G$ to $G'$ decreases the length of the pendant path $P_m$ at $w$ by exactly $1$. It follows from Lemma~\ref{DeleteAddLeaf} that $c(G')=c(G)$ and that $$\N(G')-\N(G)=\N(H)_x - \N(H)_y.$$ We are going to show that $\N(G') - \N(G)>0$. By setting $H_1=H- V(G_1-v_1)$ and  $H_2=G- V(P_m-w)$, we obtain
\begin{align*}
\N(H)_x = (n_1-1) + \N(H_1)_{v_1}\,, \quad \N(H)_y= (m-2) + \N(H_2)_w\,,
\end{align*}
and thus $ \N(G')-\N(G)=(n_1-1) - (m-2) + \N(H_1)_{v_1} - \N(H_2)_w$. Now we distinguish between two subcases depending on the order of $C$.

Subcase 1: The order of $C$ is $3$.

Let us determine an expression for $\N(H_1)_{v_1}$. It can be decomposed into
\begin{align*}
\N(H_1)_{v_1,v_2,v_0}&=n_2(t-1+(m-1)\N(A)_w)\,, \quad 
\N(H_1-v_0)_{v_1,v_2}=n_2\,,\\
\N(H_1-v_2)_{v_1}&=1+ (t-1) +(m-1)\N(A)_w\,,
\end{align*}
and thus
\begin{align*}
\N(H_1)_{v_1}=\N(H_1)_{v_1,v_2,v_0}+ \N(H_1-v_0)_{v_1,v_2} + \N(H_1-v_2)_{v_1} =(n_2+1)(t+(m-1)\N(A)_w)\,.
\end{align*}
Let us determine an expression for $\N(H_2)_w$. Denote by $Q$ the subgraph of $G$ that consists of the triangle $v_0v_1v_2$ and the pendant paths $G_1$ and $G_2$ at $v_1,v_2$, respectively. We have
\begin{align*}
\N(Q)_{v_0,v_1,v_2}=n_1\cdot n_2\,, \quad \N(Q-v_2)_{v_0,v_1}=n_1\,, \quad \N(Q-v_1)_{v_0}= 1+n_2\,,
\end{align*}
and these quantities imply that
\begin{align*}
\N(Q)_{v_0}&=\N(Q)_{v_0,v_1,v_2} + \N(Q-v_2)_{v_0,v_1} + \N(Q-v_1)_{v_0}=(n_1+1)(n_2+1)\,,\\
\N(B)_z&=t-2+\N(Q)_{v_0} \,, \quad \N(H_2)_w=\N(A)_w(t-1+\N(Q)_{v_0})\,.
\end{align*}
Therefore, we have
\begin{align*}
\N(G')-\N(G)&= (n_1-1) - (m-2) + \N(H_1)_{v_1} - \N(H_2)_w\\
&=n_1+ t(n_2+1) +(m-1)((n_2+1)\N(A)_w -1) - \N(A)_w(t-1+\N(Q)_{v_0})\,,
\end{align*}
and since
\begin{align*}
t-2+\N(Q)_{v_0}=\N(B)_z \leq m\,,
\end{align*}
we deduce that
\begin{align*}
\N(G')-\N(G)&\geq n_1+ t(n_2+1) +(m-1)((n_2+1)\N(A)_w -1) - \N(A)_w (m+1)\\
&=(m-1)(n_2\cdot \N(A)_w -1) - (2\N(A)_w -1) +(n_1-1) +t(n_2+1)\,.
\end{align*}
Note that $\N(A)_w, t\geq 2, ~n_1, n_2\geq 1$ and thus $m\geq 4$. If $n_2>1$ then 
\begin{align*}
\N(G')-\N(G)&\geq (m-1)(n_2\cdot \N(A)_w -1) - (n_2 \cdot \N(A)_w -1) +(n_1-1) +t(n_2+1)\\
&= (m-2)(n_2\cdot \N(A)_w -1)  +(n_1-1) +t(n_2+1)>0\,.
\end{align*}
If $n_2=1$ then 
\begin{align*}
\N(G')-\N(G)&\geq (m-1)(\N(A)_w -1) - (2\N(A)_w -1) +(n_1-1) +2t\\
&=(m-3)(\N(A)_w -1) +n_1 +2(t-1)>0\,.
\end{align*}
For either situation, $G'$ contains more connected induced subgraphs than $G$, a contradiction to the choice of $G$. 

Subscase~2: The order of $C$ is $4$.

Denote by $n_3$ the order of the pendant path $G_3$ at $v_3$. We compute $\N(H_1)_{v_1}$ and $\N(H_2)_w$ in a similar way as in Subcase~1. The quantity $\N(H_1)_{v_1}$ can be decomposed into
\begin{align*}
\N(H_1)_{v_1,v_2,v_3,v_0}&=n_2\cdot n_3(t-1+(m-1)\N(A)_w)\,, \quad 
\N(H_1-v_0)_{v_1,v_2,v_3}=n_2\cdot n_3\,,\\
 \N(H_1-v_3)_{v_1,v_2}&=n_2(1+ (t-1) +(m-1)\N(A)_w)\,, \\
  \N(H_1-v_2)_{v_1}&=1+ (1+n_3)((t-1) +(m-1)\N(A)_w)\,,
\end{align*}
and thus
\begin{align*}
\N(H_1)_{v_1}&=\N(H_1)_{v_1,v_2,v_3,v_0}+ \N(H_1-v_0)_{v_1,v_2,v_3}+ \N(H_1-v_3)_{v_1,v_2} + \N(H_1-v_2)_{v_1}\\
&=n_2\cdot n_3 +n_2+1 + (t-1 +(m-1)\N(A)_w)(n_2\cdot n_3+n_2+n_3+1)\,.
\end{align*}
Denote by $Q$ the subgraph of $G$ that consists of the square $v_0v_1v_2v_3$ and the pendant paths $G_1,G_2,G_3$ at $v_1,v_2,v_3$, respectively. We have
\begin{align*}
\N(Q)_{v_0,v_1,v_2,v_3}&=n_1\cdot n_2\cdot n_3\,, \quad \N(Q-v_3)_{v_0,v_1,v_2}=n_1\cdot n_2\,,\\
\N(Q-v_2)_{v_0,v_1}&= n_1(1+n_3)\,, \quad \N(Q-v_1)_{v_0}= 1+ n_3(1+n_2)\,,
\end{align*}
and these quantities imply that
\begin{align*}
\N(Q)_{v_0}&=(n_1+1)(n_2+1)(n_3+1) - n_2\,, \quad \N(B)_z=t-2+\N(Q)_{v_0}\,, \\ \quad \N(H_2)_w&=\N(A)_w(t-1+\N(Q)_{v_0})\,,
\end{align*}
and
\begin{align*}
\N(G')-\N(G)=&~(n_1-1) - (m-2) + \N(H_1)_{v_1} - \N(H_2)_w\\
=&~n_2\cdot n_3 +n_2+n_1 +1 +(t-1)(n_2\cdot n_3 +n_2+n_3+1) \\
&+ (m-1)((n_2\cdot n_3 +n_2+n_3 +1)\N(A)_w -1) - \N(A)_w(t-1+\N(Q)_{v_0})\,.
\end{align*}
Since
\begin{align*}
t-2+\N(Q)_{v_0}=\N(B)_z \leq m\,,
\end{align*}
we deduce that
\begin{align*}
\N(G')-\N(G)\geq&~ n_2\cdot n_3 +n_2+n_1 +1 + (t-1)(n_2\cdot n_3 +n_2+n_3+1)\\
&+(m-1)((n_2\cdot n_3 +n_2+n_3 +1)\N(A)_w -1) - \N(A)_w(m+1)\\
=&~n_2\cdot n_3 +n_2+n_1 +(t-1)(n_2\cdot n_3 +n_2+n_3+1)\\
& + (m-1)((n_2\cdot n_3 +n_2+n_3)\N(A)_w -1) -(2 \N(A)_w -1)\,.
\end{align*}
Using the inequality $n_2\cdot n_3 +n_2+n_3>2$, we derive that
\begin{align*}
\N(G')-\N(G) >&~ n_2\cdot n_3 +n_2+n_1 +(t-1)(n_2\cdot n_3 +n_2+n_3+1)\\
& + (m-2)((n_2\cdot n_3 +n_2+n_3)\N(A)_w -1)>0\,.
\end{align*}
This is again a contradiction to the choice of $G$.

\medskip
Summing up, we have proved that $w\neq v_0$ is impossible. Hence $w$ must belong to the cycle of $G$.
\end{proof}

Let us summarise in the following definition what we already know about the structure of those maximal graphs in $\mathcal{U}(n,c)$.

\begin{definition}\label{FirstSummary}
Let $G\in \mathcal{U}(n,c)$ such that $G$ has the maximum number of connected induced subgraphs. By Proposition~\ref{Prop:Girth3or4} the order of the cycle $C$ of $G$ is $3$ or $4$. 
\begin{itemize}
\item If $C$ is of order $3$, then we denote by $v_0,v_1,v_2$ the vertices of $C$ in this order, and by $G_0,G_1,G_2$ the components of $G-\{v_0v_1,v_1v_2,v_2v_0\}$ that contain $v_0,v_1,v_2$, respectively.
\item If $C$ is of order $4$, then we denote by $v_0,v_1,v_2,v_3$ the vertices of $C$ in this order, and by $G_0,G_1,G_2,G_3$ the components of $G-\{v_0v_1,v_1v_2,v_2v_3,v_3v_0\}$ that contain $v_0,v_1,v_2,v_3$, respectively. In this case, by Proposition~\ref{Prop:Girth3or4} it holds that $$\N(G_2)_{v_2} =\max_{1\leq j \leq 4} \N(G_j)_{v_j} ~~\text{and that} ~~ \N(G_2)_{v_2} \leq \N(G_3)_{v_3}(1+\N(G_1)_{v_1}).$$
\end{itemize}
If $G$ has a branching vertex, then by Propositions~\ref{Prop:onebranchvertex} and ~\ref{branchvertcycle} this vertex is unique and must lie on $C$. Suppose that $v_2$ is the branching vertex of $G$. Then depending on the order of $C$, the graphs $G_0,G_1$ are pendant paths at $v_0,v_1$, respectively, and the graphs $G_0,G_1,G_3$ are pendant paths at $v_0,v_1,v_3$, respectively.
\end{definition}

Unless otherwise specified, we follow the notation given in Definition~\ref{FirstSummary}.

\subsection{Pendant paths have orders' difference at most $1$}\label{Subsec:almostequalpendantpath}
The first part of the next lemma is a refinement of~\cite[Lemma~6]{AudaceCut2019}. Lemma~\ref{PathOrder} below gives some information about the pendants paths at adjacent vertices of the cycle of every maximal graph.

\begin{lemma}\label{PathOrder}
Let $H$ be a connected graph and $uv$ an edge of $H$. Let $H(n_1;n_2)$ be the graph obtained from $H$ and two vertex disjoint paths $P_{n_1}$ and $P_{n_2}$ by identifying $u$ with an endvertex of $P_{n_1}$, and $v$ with an endvertex of $P_{n_2}$. Assume that $1\leq n_1\leq n_2-2$ and that $u$ is not a pendant edge of $H$. Then the following hold:
	\begin{enumerate}[i)]
		\item We have
		\begin{align*}
		\N(H(n_1;n_2)) < \N(H(n_1+1;n_2-1))\,.
		\end{align*}
		\item Furthermore, we have $$c(H(n_1+1;n_2-1))=c(H(n_1;n_2))\,,$$ unless $n_1=1$ and $u$ is a cut vertex of $H$.
	\end{enumerate}	
\end{lemma}

\begin{proof}
It was shown in the proof of~\cite[Lemma~6]{AudaceCut2019} that
	\begin{align*}
	\N(H(n_1;n_2))=&~n_1\cdot n_2\cdot \N(H)_{u,v}+n_1\cdot \N(H-v)_u + n_2\cdot \N(H-u)_v\\
	&+ \binom{n_1}{2}+  \binom{n_2}{2} +\N(H-u-v)\,.
	\end{align*}
	Thus 
	\begin{align*}
	\N(H(n_1;n_2))- &\N(H(n_1+1;n_2-1))\\
	&=(n_1-n_2+1)(\N(H)_{u,v} -1) +  \N(H-u)_v -  \N(H-v)_u\,,
	\end{align*}
obtained after simplification. On the other hand, one can add the edge $uv$ to every $v$-containing connected induced subgraph of $H-u$ to obtain a $u$-containing connected subgraph of $H$, which can then be extended to an induced subgraph of $H$. Therefore we get
	\begin{align*}
	\N(H)_{u,v} + \N(H-v)_u = \N(H)_u\geq \N(H-u)_v +2\,,
	\end{align*}
	where the final $2$ counts the subgraphs $u$ and $uw$ ($w \neq v$ is a neighbour of $u$ in $H$). This implies that $\N(H-u)_v -  \N(H-v)_u \leq \N(H)_{u,v} -2$. Hence
	\begin{align*}
	\N(H(n_1;n_2))- \N(H(n_1+1;n_2-1))\leq (n_1-n_2+2)(\N(H)_{u,v} -1) -1 <0\,,
	\end{align*}
	which proves i). We now consider proving ii). If $n_1>1$, then it is clear that the number of cut vertices is preserved when passing from $H(n_1;n_2)$ to $H(n_1+1;n_2-1)$. Assume that $n_1=1$ and that $u$ is not a cut vertex of $H$ (thus of $H(n_1;n_2)$). Denote by $w$ the free endvertex of $P_{n_2}$ in $H(1;n_2)$, and by $w'$ the neighbour of $w$. Then the graph $H(2;n_2-1)$ can be constructed from $H(1;n_2)$ by contracting $w'$ and adding a pendant edge $uw'$. By so doing $w'$ (resp. $u$) becomes a non cut vertex (resp. cut vertex) of $H(2;n_2-1)$. Hence $c(H(2;n_2-1))=c(H(1;n_2))$.
\end{proof}

\begin{proposition}\label{Prop:almostequallength}
If $G$ is a maximal graph, then the orders of the pendants paths at
\begin{enumerate}[i)]
\item  adjacent vertices of $C$ must be as equal as possible;
\item  the branching vertex (if any) of $G$ must be as equal as possible.
\end{enumerate}
\end{proposition}

\begin{proof}
\begin{enumerate}[i)]
\item  Let $u,v$ be two adjacent vertices of $C$, and $P_{n_1}, P_{n_2}$ be pendant paths at $u,v$, respectively such that $n_1\leq n_2$. Set $H=G-(V(P_{n_1}-u) \cup V(P_{n_2}-v))$. Then both $u$ and $v$ are non cut vertices of $H$. By Lemma~\ref{PathOrder} the inequality $n_1\leq n_2-2$ is impossible since otherwise, a new graph $G'$ satisfying $\N(G')>\N(G)$ and $c(G')=c(G)$ could be constructed from $G$. Hence we must have $n_2-1 \leq n_1 \leq n_2$, that is $|n_1-n_2|\leq 1$ holds.
\item Let $w$ be the branching vertex of $G$ and $P_{n_1}, P_{n_2}$ two pendant paths at $w$. Then $n_1,n_2>1$, and we can assume that $n_1\leq n_2$. If $n_2=2$ then $n_1=2$ as well, and we are done in this case. Otherwise $n_2>2$. By setting $A=G-(V(P_{n_1}-w) \cup V(P_{n_2}-w))$ and invoking Remark~\ref{Rem:pendantSingleVertex} (see Subsection~\ref{Subsec:branchvertexcycle}), we obtain $|n_1-n_2|\leq 1$, which completes the proof.
\end{enumerate}
\end{proof}

\begin{proposition}\label{branchminorder}
Suppose that $G$ is a maximal graph whose branching vertex is $v_0$. Assume that $n_0$ is the minimum order among the pendants paths at $v_0$. Then the order of every pendant path at a neighbour of $v_0$ in $C$ is at most $n_0$.
\end{proposition}

\begin{proof}
Denote by $x$ the free endvertex of the pendant path $P_{n_0}$ at $v_0$, and by $u_0$ (possibly $u_0=x$) the neighbour of $v_0$ on the pendant path $P_{n_0}$. Let $A$ be the component of $G_0-u_0v_0$ that contains $v_0$.
	
Case 1: $C$ is of order $3$.

Denote by $n_1,n_2$ the orders of the pendant paths $G_1,G_2$ at $v_1,v_2$, respectively. Assume without loss of generality that $n_2\leq n_1$. Suppose to the contrary that $n_1>n_0$. Then $n_1=n_0+1>2$ by virtue of Proposition~\ref{Prop:almostequallength}. Let $u$ be the free endvertex of $G_1$ at $v_1$, and $y$ the neighbour of $u$ in $G$. Let $G'$ be the graph constructed from $G$ by deleting the edge $uy$ and adding the edge $ux$. Set $H=G-u$, and note that both $x$ and $y$ are pendant vertices of $H$. Then Lemma~\ref{DeleteAddLeaf} yields $c(G')=c(G)$ and $\N(G')-\N(G)=\N(H)_x - \N(H)_y$.

Let us derive an expression for both $\N(H)_x$ and $\N(H)_y$. Denote by $Q$ the subgraph of $G$ that consists of the triangle $v_0v_1v_2$ and the pendant paths $G_1-u$ and $G_2$. Thus $\N(Q)$ can be decomposed into
\begin{align*}
\N(Q)_{v_0,v_1,v_2}=n_2(n_1-1)\,, \quad \N(Q-v_2)_{v_0,v_1}=n_1-1\,, \quad \N(Q-v_1)_{v_0}=1+n_2\,,
\end{align*}
and these quantities imply that
\begin{align*}
\N(Q)_{v_0}=n_2(n_1-1)+(n_1-1) +1+n_2\,,\quad \N(H)_x=(n_0-1)+\N(A)_{v_0} \N(Q)_{v_0}\,.
\end{align*}
Likewise, by setting $H'=G-V(G_1-v_1)$, we can decompose $\N(H')$ into the following quantities:
\begin{align*}
\N(H')_{v_1,v_0,v_2}=n_2\cdot n_0\cdot \N(A)_{v_0}\,, \quad \N(H'-v_2)_{v_1,v_0}=n_0\cdot \N(A)_{v_0}\,, \quad \N(H'-v_0)_{v_1}=1+n_2\,.
\end{align*}
Thus
\begin{align*}
\N(H')_{v_1}=(n_0\cdot n_2+n_0)\N(A)_{v_0} +1+n_2\,,\quad \N(H)_y=(n_1-2)+\N(H')_{v_1}\,.
\end{align*}
Since $n_1=n_0+1$, it follows (after simplification) that
\begin{align*}
\N(G')-\N(G)=&\N(H)_x - \N(H)_y\\
=&(n_0-1)+\N(A)_{v_0} \N(Q)_{v_0} - (n_1-2) - \N(H')_{v_1}\\
=&(1+n_2)(\N(A)_{v_0} -1)  >0\,,
\end{align*}
a contradiction to the maximality of $G$.

Case~2: $C$ is of order $4$.

The reasoning is essentially the same as in Case~1 with a minor revision of the notation. Denote by $n_1,n_2,n_3$ the orders of the pendant paths $G_1,G_2,G_3$ at $v_1,v_2,v_3$, respectively. Assume without loss of generality that $n_3\leq n_1$. Suppose to the contrary that $n_1>n_0$. Then $n_1=n_0+1>2$ by virtue of Proposition~\ref{Prop:almostequallength}. Let $u$ be the free endvertex of the pendant path $G_1$ at $v_1$, and $y$ the neighbour of $u$ in $G$. Construct a new graph $G'$ from $G$ by deleting the edge $uy$ and adding the edge $ux$. By setting $H=G-u$ and invoking Lemma~\ref{DeleteAddLeaf}, we get $c(G')=c(G)$ and $\N(G')-\N(G)=\N(H)_x - \N(H)_y$. Now we consider evaluating $\N(H)_x- \N(H)_y$.

Denote by $Q$ the subgraph of $G$ that consists of the square $v_0v_1v_2v_3$ and the pendant paths $G_1-u, G_2$ and $G_3$. We have
\begin{align*}
\N(Q)_{v_0,v_1,v_2,v_3}&=n_2\cdot n_3(n_1-1)\,, \quad \N(Q-v_3)_{v_0,v_1,v_2}=n_2(n_1-1)\,, \\
 \N(Q-v_2)_{v_0,v_1}&=(n_1-1)(1+n_3)\,, \quad \N(Q-v_1)_{v_0}=1+n_3(1+n_2)\,,
\end{align*}
and thus it holds that
\begin{align*}
\N(Q)_{v_0}&=(n_1-1)(n_2 \cdot n_3+n_2+1+n_3)+1+n_3(1+n_2)\,,\\
\N(H)_x&=(n_0-1)+\N(A)_{v_0} \N(Q)_{v_0}\,.
\end{align*}
Likewise, by setting $H'=G-V(G_1-v_1)$, the quantity $\N(H')$ can be decomposed into
\begin{align*}
\N(H')_{v_1,v_0,v_3,v_2}&=n_2\cdot n_3 \cdot n_0\cdot \N(A)_{v_0}\,, \quad \N(H'-v_2)_{v_1,v_0,v_3}=n_3\cdot n_0 \cdot \N(A)_{v_0}\,, \\ \N(H'-v_3)_{v_1,v_0}&=(n_2+1) n_0 \cdot \N(A)_{v_0}\,, \quad \N(H'-v_0)_{v_1}=1+n_2(1+n_3)\,,
\end{align*}
and thus
\begin{align*}
\N(H')_{v_1}&=(n_2 \cdot n_3+n_3+n_2+1)n_0\cdot \N(A)_{v_0} + 1+n_2(1+n_3)\,,\\
\N(H)_y&=(n_1-2)+\N(H')_{v_1}\,.
\end{align*}
Since $n_1=n_0+1$ and $n_2\leq 1+n_3$ (see Proposition~\ref{Prop:almostequallength}), it follows that
\begin{align*}
\N(G')-\N(G)&= (n_0-1)+\N(A)_{v_0} \N(Q)_{v_0} - (n_1-2) - \N(H')_{v_1}\\
&= (1+n_3 +n_3\cdot n_2)\N(A)_{v_0} - (1+n_2+n_3\cdot n_2)\\
&\geq  2(1+n_3 +n_3\cdot n_2) - (1+n_2+n_3\cdot n_2)\\
&= (1+n_3)-n_2 +n_3(n_2+1)>0\,.
\end{align*}
Hence, a contradiction to the choice of $G$.

This completes the proof of the proposition.
\end{proof}

We recall that $n>3$ and $0<c<n-2$ are fixed integers and that $G\in \mathcal{U}(n,c)$ is an arbitrary unicyclic graph with order $n$ and $c$ cut vertices that has the maximum number of connected induced subgraphs. From here onwards, we consistently assume that $v_2$ is the branching vertex (if any) of $G$. 

\medskip
We are now ready to characterise those maximal graphs with girth $3$. 

\subsection{Maximal graphs with girth $3$}\label{Subsec:girth3structure}
The following description for the graph $\Delta_{n,c}$ is also given in Section~\ref{Intro:main}.
\begin{definition}
Let $r$ be the residue of $n-3$ modulo $n-c$ and $q=\lfloor (n-3)/(n-c) \rfloor$. Set $m_j=q+1$ for all $1\leq j \leq r$, and $m_j=q$ for all $r+1\leq j \leq n-c$. Then we define $\Delta_{n,c}$ to be the graph constructed from the triangle $v_0v_1v_2$ by attaching $n-c-2$ pendant paths of respective lengths $m_1,m_2,\ldots,m_{n-c-2}$ at $v_2$, one pendant path of length $m_{n-c-1}$ at $v_1$, and one pendant path of length $m_{n-c}$ at $v_0$.
\end{definition}
Note that $\Delta_{n,c}$ has $n$ vertices of which $c$ are cut vertices.

\begin{proposition}\label{Prop:Delta3}
If the girth of $G$ is $3$, then $G$ is isomorphic to $\Delta_{n,c}$.
\end{proposition}

\begin{proof}
Denote by $n_0,n_1$ the orders of the pendant paths $G_0,G_1$ at $v_0,v_1$, respectively.
	
Case~1: $G$ has no branching vertex.
	
Then $G_2$ is also a pendant path at $v_2$, and thus $n-c=3$. So we can assume that $|V(G_2)|\geq n_1,n_0$. By Proposition~\ref{Prop:almostequallength}, $n_1,n_0 \geq |V(G_2)| -1$. Hence $G$ is isomorphic to $\Delta_{n,c}$.
	
Case~2: $v_2$ is the branching vertex of $G$.
	
Let $n_2$ (resp. $n_2'$) be the minimum (resp. maximum) order among the pendant paths at $v_2$. Then Propositions~\ref{Prop:almostequallength} and~\ref{branchminorder} yield $n_2'-1\leq n_0,n_1 \leq n_2$ and $n_2\leq n_2' \leq n_2+1$. These inequalities imply that $n_0=n_1 =n_2$ if $n_2'=n_2+1$, and $n_0,n_1\in \{n_2-1,n_2\}$ if $n_2'=n_2$. Hence $G$ is isomorphic to $\Delta_{n,c}$.
\end{proof}

\subsection{Maximal graphs with girth $4$}\label{Subsec:girth4structure}
We require a final lemma prior to characterising those maximal graphs with girth $4$.

\begin{lemma}\label{Twopendantgirth4Order}
Assume that the girth of $G$ is $4$, and that $G$ has a branching vertex. Then there are precisely two pendant paths at $v_2$. Moreover, it holds that $$|V(G_3)|=|V(G_1)|=n_2,$$ where $n_2$ is the minimum order among the pendants paths at $v_2$.
\end{lemma}

\begin{proof}
Denote by $n_0,n_1,n_3$ the orders of the pendant paths at $v_0,v_1,v_3$, respectively. By Proposition~\ref{Prop:almostequallength}, $|n_j-n_{j+1}|\leq 1$ for all $j\in \{0,1,2,3\}$, where $n_4=n_0$. So we have
\begin{align*}
\N(G_2)_{v_2}\geq n_2^2\geq n_2+2 \geq n_0,n_1,n_3\,,
\end{align*}
and thus $\N(G_2)_{v_2}=\max_{0\leq j \leq 3} \N(G_j)_{v_j}$. Also recall that  $n_1,n_3\leq n_2$ by virtue of Proposition~\ref{branchminorder}. Therefore, if there are more than two pendant paths at $v_2$, then $$\N(G_2)_{v_2}\geq n_2^3>n_2 (1+n_2)\geq n_3(1+n_1)=\N(G_3)_{v_3} (1+\N(G_1)_{v_1}).$$ However, this strict inequality is impossible (see the summary in Definition~\ref{FirstSummary}). Hence there are only two pendant paths at $v_2$.

Now assume that there are only two pendant paths at $v_2$, and without loss of generality, say $n_1\geq n_3$. Recall that $n_2-1 \leq n_3 \leq n_2$ and $n_3\leq n_1 \leq n_2$. If $n_3=n_2-1$, then $$\N(G_2)_{v_2}\geq n_2^2>(n_2-1)(1+n_2)\geq n_3(1+n_1)=\N(G_3)_{v_3} (1+\N(G_1)_{v_1}),$$ which is again a contradiction to the choice of $G$ (see Definition~\ref{FirstSummary}). Hence we must have $n_3= n_2$ and thus $n_1=n_2$. This completes the proof of the lemma.
\end{proof}

The following description for the graphs $\Omega_{n,n-4}$  and $\Omega_{n,n-5}$ are also given in Section~\ref{Intro:main}.

\begin{definition}
Define the graphs $\Omega_{n,n-4}$ and $\Omega_{n,n-5}$ as follows:
\begin{itemize}
	\item Set $m=\lfloor n/4 \rfloor$ and let $r$ be the residue of $n$ modulo $4$. Then $\Omega_{n,n-4}$ is the graph with order $n$ and $n-4$ cut vertices constructed from the square $v_0v_1v_2v_3$ by attaching the pendant paths of orders $m_0,m_1,m_2,m_3$ at $v_0,v_1,v_2,v_3$, respectively, where $(m_0,m_1,m_2,m_3)$ is equal to
	\begin{align*}
	(m,m,m,m),(m+1,m,m,m),(m+1,m+1,m,m),(m+1,m+1,m+1,m)
	\end{align*}
	when $r$ is equal to $0,1,2,3$, respectively.
	\item Let $n>7$ such that $n+k=5m$ for some integer $m>1$ and some $k\in \{0,1,2\}$. Then $\Omega_{n,n-5}$ is the graph with order $n$ and $n-5$ cut vertices constructed from the square $v_0v_1v_2v_3$ by attaching the pendant paths of orders $m_0,m_1,m_2,m_3$ at $v_0,v_1,v_2,v_3$, respectively, and another pendant of order $m$ at $v_2$, where $(m_0,m_1,m_2,m_3)$ is equal to
	\begin{align*}
	(m,m,m+1,m),(m,m,m,m),(m-1,m,m,m)
	\end{align*}
	when $k$ is equal to $0,1,2$, respectively.
\end{itemize}
\end{definition}

\begin{proposition}\label{Prop:Omega4}
If the girth of $G$ is $4$, then $G$ is isomorphic to $\Omega_{n,n-5}$ or $\Omega_{n,n-4}$.
\end{proposition}

\begin{proof}
Denote by $n_0,n_1,n_3$ the orders of the pendant paths $G_0,G_1,G_3$ at $v_0,v_1,v_3$, respectively.

Case 1: $G$ has no branching vertex.

Then $G_2$ is also a pendant path at $v_2$, and thus $c=n-4$. So we can assume that $ n_0 \geq n_1,|V(G_2)|,n_3$, and thus $n_0>1$. Set $n_2=|V(G_2)|$ and note that by Proposition~\ref{Prop:almostequallength},
\begin{align*}
n_0-1 \leq n_1,n_3 \leq n_0 \quad \text{and} \quad n_0-2 \leq n_1-1 \leq n_2 \leq n_0\,.
\end{align*}
{\sc Claim:} $n_2 \neq n_0-2$.

Suppose to the contrary that $n_2= n_0-2$. Then $n_0>2$. Let $u$ be the free endvertex of the pendant path $G_0$ at $v_0$, and $y$ the neighbour of $u$ in $G$. Likewise, let $x$ be the free endvertex of the pendant path $G_2$ at $v_2$. Delete the edge $uy$ and add the new edge $ux$ to obtain the graph $G'$. Then Lemma~\ref{DeleteAddLeaf} yields 
\begin{align*}
c(G')=c(G) \quad \text{and} \quad \N(G') -\N(G)=\N(G-u)_x-\N(G-u)_y\,.
\end{align*}
Let us estimate the difference $\N(G-u)_x-\N(G-u)_y$. To achieve this, set $L=(G-u)-V(G_2-v_2)$ and $L'=G-V(G_0 - v_0)$. Since $n_0-2=n_2$, one notices that $\N(L')_{v_0}=\N(L-y)_{v_2}$. Moreover, it holds that
\begin{align*}
\N(G-u)_x&=(n_0-3) + \N(L)_{v_2}\,, \quad \N(G-u)_y=(n_0-2) + \N(L')_{v_0}\,,\\
\N(L)_{v_2}&=\N(L-y)_{v_2} + \N(L)_{v_2,y}\,,
\end{align*}
from which we deduce that $\N(G-u)_x-\N(G-u)_y = \N(L)_{v_2,y}- 1>0$. This contradiction proves the claim.

Hence we must have $n_0-1 \leq n_1,n_2, n_3 \leq n_0$. Then the sequence $(n_0,n_1,n_2,n_3)$ consists of at most two distinct values. Therefore, this sequence defines $G$ uniquely unless $n_0$ appears exactly two times, in which case there are only two possibilities for $G$. Assume that $|V(G_j)|=n_0$ for some $j\neq 0 $. Let $H_1$ (resp. $H_2$) be the graph corresponding to the situation where $v_j$, say $v_1$ and $v_0$ are adjacent (resp. $v_j$ and $v_0$ are not adjacent) in $G$. We show that $\N(H_1) > \N(H_2)$.

Denote by $u_1$ (resp. $u_2$) the free endvertex of the pendant path $G_1$ at $v_1$ in $H_1$ (resp. the pendant path $G_2$ at $v_2$ in $H_2$). Then $H_1-u_1$ and $H_2-u_2$ are isomorphic graphs. By setting $L_1=H_1-V(G_1-v_1)$ and $L_2=H_2-V(G_2-v_2)$, we get the decomposition
\begin{align*}
\N(L_1)_{v_1,v_2,v_3,v_0}&=n_0(n_0-1)^2\,, \quad \N(L_1-v_0)_{v_1,v_2,v_3}=(n_0-1)^2\,, \\
\N(L_1-v_3)_{v_1,v_2}&=(n_0-1)(n_0+1)\,, \quad \N(L_1-v_2)_{v_1}=1+ n_0(1+n_0-1)
\end{align*}
for $\N(L_1)_{v_1}$, and the decomposition
\begin{align*}
\N(L_2)_{v_2,v_3,v_0,v_1}&=n_0(n_0-1)^2\,, \quad \N(L_2-v_1)_{v_2,v_3,v_0}=n_0(n_0-1)\,, \\
\N(L_2-v_0)_{v_2,v_3}&=(n_0-1)(1+n_0-1)\,, \quad \N(L_2-v_3)_{v_2}=1+ (n_0-1)(1+n_0)
\end{align*}
for $\N(L_2)_{v_2}$. Direct calculations show that $\N(L_1)_{v_1} - \N(L_2)_{v_2}=1$. On the other hand, 
\begin{align*}
\N(H_1)_{u_1}=(n_0-1)+\N(L_1)_{v_1}\,, \quad \N(H_2)_{u_2}=(n_0-1)+\N(L_2)_{v_2}\,,
\end{align*}
and these identities imply $$\N(H_1)-\N(H_2)=\N(H_1)_{u_1} -\N(H_2)_{u_2}=\N(L_1)_{v_1} - \N(L_2)_{v_2}=1>0.$$ Hence $G$ must be isomorphic to $H_1$. It is easy to see that $H_1$ is isomorphic to $\Omega_{n,n-4}$.

Case 2: $v_2$ is the branching vertex of $G$.

Denote by $n_2$ (resp. $n_2'$) the minimum (resp. maximum) order among the pendant paths at $v_2$. Then $1<n_2 \leq n_2'$. By Lemma~\ref{Twopendantgirth4Order} there are precisely two pendant paths at $v_2$, and $n_3=n_1=n_2$. Thus $c=n-5$. By Definition~\ref{FirstSummary} it holds that $$n_2\cdot n_2'=\N(G_2)_{v_2}\leq \N(G_3)_{v_3}(1+\N(G_1)_{v_1})=n_2(1+n_2),$$ that is $n_2\leq n_2'\leq 1+ n_2$. Moreover, if $G'$ is the graph constructed from $G$ in Lemma~\ref{Girth3or4} (see Subsection~\ref{Subsec:girth3or4}), then $\N(G)= \N(G')$ holds if and only if $n_2'= 1+ n_2$.
\begin{itemize}
\item Assume that $n_2'= 1+ n_2$. Then $G'$ must be isomorphic to $\Delta_{n,n-5}$ by virtue of Proposition~\ref{Prop:Delta3}. In particular, we get $n_0=n_2$. Therefore $n=5n_2$ and $G$ is isomorphic to the graph $\Omega_{n,n-5}$.
\item Assume that $n_2'= n_2$. Proposition~\ref{Prop:almostequallength} yields $n_2-1\leq n_0 \leq n_2+1$. If $n_0 = n_2+1$, then $n=5n_2$ and direct calculations give $$\N(\Omega_{n,n-5}) - \N(G)=(n_2+1)^2(n_2-1)>0,$$ a contradiction to the choice of $G$. If $n_0=n_2$, then $n=5n_2-1$ and $G$ is isomorphic to $\Omega_{n,n-5}$. If $n_0=n_2-1$, then $n=5n_2-2$ and $G$ is also isomorphic to $\Omega_{n,n-5}$.
\end{itemize}
This completes the proof of the proposition.
\end{proof}

We are now prepared to prove our main theorem.

\subsection{Proof of the main theorem}\label{Subsec:maintheorem}
In this subsection, we present a proof of our main theorem. We first recall the main result stated in Section~\ref{Intro:main}.
\begin{theorem}\label{MainTheoremFinal}
Let the integers $n>3$ and $0< c<n-2$ be given, and $G\in \mathcal{U}(n,c)$ such that $\N(G)\geq \N(H)$ for all $H\in \mathcal{U}(n,c)$. Then the following hold:
\begin{itemize}
	\item $G$ is only isomorphic to $\Delta_{n,c}$ if $c=n-3$ or $c<n-5$;
	\item $G$ is only isomorphic to $\Omega_{n,c}$ if $c=n-4>1$, or $c=n-5>3$ and $n=3,4\mod 5$;
	\item $G$ is isomorphic to both $\Delta_{n,c}$ and $\Omega_{n,c}$ if $c=n-5=3$, or $c=n-4 =1$, or
	\begin{align*}
	\textrm{
		$c=n-5>0$ and $ n=0\mod 5$.}
	\end{align*}
\end{itemize}

\end{theorem}

\begin{proof}
Since the girth of $G$ is $3$ or $4$, it suffices to compare $\N(\Delta_{n,c})$ with $\N(\Omega_{n,c})$ depending on the values of $n$ and $c$; see Propositions~\ref{Prop:Delta3} and~\ref{Prop:Omega4}. Recall that $\Omega_{n,c}$ has precisely $c=n-5$ or $c=n-4$ cut vertices.

Case~1: $c=n-3$ or $c<n-5$. 

Clearly, $G$ is only isomorphic to $\Delta_{n,c}$.

Case~2: $c=n-4>0$.

If $n=4m$ for some integer $m>1$, or $n=4m+3$ for some integer $m>0$, then the case $g=4$ in Lemma~\ref{Girth3or4} gives us $\N(\Delta_{n,n-4})< \N(\Omega_{n,n-4})$. Thus $G$ is only isomorphic to $\Omega_{n,n-4}$ in this case. On the other hand, using Lemma~\ref{Formula3or4} (see Subsection~\ref{Subsec:girth3or4}) direct calculations show that $$\N(\Omega_{n,n-4})-\N(\Delta_{n,n-4})=m(m^2+m-1)>0$$ if $n=4m+2$ for some integer $m>0$, and $$\N(\Omega_{n,n-4})-\N(\Delta_{n,n-4})=m(m-1)(m+1)$$ if $n=4m+1$ for some integer $m>0$. The latter identity shows that $\N(\Omega_{n,n-4})>\N(\Delta_{n,n-4})$ for $m>1$, and $\N(\Omega_{n,n-4})=\N(\Delta_{n,n-4})$ for $m=1$.

Case~3: $c=n-5>0$.

Then we infer from the case $g=4$ in Lemma~\ref{Girth3or4} that $\N(\Delta_{n,n-5})=\N(\Omega_{n,n-5})$ if $n=5m$ for some integer $m>1$, and $\N(\Delta_{n,n-5})< \N(\Omega_{n,n-5})$ if $n=5m-1$ for some integer $m>1$. Using Lemma~\ref{Formula3or4}, direct calculations yield $$\N(\Omega_{n,n-5})-\N(\Delta_{n,n-5})=m(m-2)\geq 0$$ if $n=5m-2$ for some integer $m>1$. The latter identity shows that $\N(\Omega_{n,n-5})>\N(\Delta_{n,n-5})$ for $m>2$, and $\N(\Omega_{n,n-5})=\N(\Delta_{n,n-5})$ for $m=2$.

This completes the proof of the theorem.
\end{proof}

\begin{remark}
Note that explicit expressions for $\N(\Delta_{n,c})$ and $\N(\Omega_{n,c})$ can easily be obtained using Lemma~\ref{Formula3or4}.
\end{remark}

\section{Concluding remarks}\label{Sec:Conclude}
The Wiener index of a connected graph is the sum of distances between all unordered vertex pairs. It is the oldest and also the most studied index among the so-called topological indices in mathematical chemistry~\cite{Entringer1976,knor2015mathematical,Wiener1947}.

It is remarkable that for the case $c=n-3$ or $c<n-5$, the same graph $\Delta_{n,c}$ minimises the Wiener index among all graphs in $\mathcal{U}(n,c)$; see~\cite[Theorem~4.7]{tan2017wiener}. It is notable, however, that this negative correlation between the Wiener index and the number of connected induced subgraphs does not extend to the other remaining cases of $n,c$. For example, if $c=n-5>3$ and $n=3,4\mod 5$, then $\Delta_{n,c}$ uniquely minimises the Wiener index~\cite[Theorem~4.7]{tan2017wiener}, while $\Omega_{n,c}$ uniquely maximises the number of connected induced subgraphs (see Theorem~\ref{MainTheoremFinal}). This is in contrast with other topological indices~\cite{WagnerCorr}.

\medskip
Naturally, what remains for further study is the analogous minimisation problem:

\begin{problem}
Characterise those graphs in $\mathcal{U}(n,c)$ with the smallest number of connected induced subgraphs.
\end{problem}
The case $c=0$ is trivial, while the case $c=1$ is easy.


\begin{thebibliography}{10}
	
\bibitem{Alokshiya2019}
M.~Alokshiya, S.~Salem, and F.~Abed.
\newblock A linear delay algorithm for enumerating all connected induced subgraphs.
\newblock {\em BMC Bioinformatics}, 20(319), DOI: \url{10.1186/s12859-019-2837-y}, 2019.

\bibitem{Chung1981}
F.~R.~K.~Chung, R.~L.~Graham, and D.~Coppersmith.
\newblock On trees containing all small trees.
\newblock {\em The Theory of Applications of Graphs},  G. Chartrand, (Editor), John Wiley and Sons, pp. 265--272, 1981.

\bibitem{GreedyWagner}
E.~O.~D.~Andriantiana, S.~Wagner and H.~Wang.
\newblock Greedy trees, subtrees and antichains
\newblock {\em Electron. J. Combin.}, 20(3), \#P28, 2013.

\bibitem{Avis2015}
D.~Avis and K.~Fukuda.
\newblock Reverse search for enumeration.
\newblock {\em Discr. Appl. Math.}, 65(1):21--46, 1996.

\bibitem{bjorklund2012traveling}
A.~Bj{\"o}rklund, T.~Husfeldt, P.~Kaski, and M.~Koivisto.
\newblock The traveling salesman problem in bounded degree graphs.
\newblock In: L.~Aceto, I.~Damg\aa rd, L.~A.~Goldberg, M.~M.~Halld\'{o}rsson, A.~Ing\'{o}lfsd\'{o}ttir, I.~Walukiewicz (eds) {\em Automata, Languages and Programming}. ICALP 2008. Lecture Notes Comput. Sci., vol 5125. Springer, Berlin, Heidelberg.

\bibitem{Audacegenral2018}
    A.~A.~V.~Dossou-Olory.
    \newblock Graphs and unicyclic graphs with extremal connected subgraphs.
    \newblock Preprint, ArXiv:1812.02422, 2018 (Submitted).
    
    \bibitem{AudaceGirth2018}
    A.~A.~V.~Dossou-Olory.
    \newblock Maximising the number of connected induced subgraphs of unicyclic graphs.
    \newblock hal-02449489 (2020), (Submitted, 2018).
	
	\bibitem{AudaceCut2019}
	A.~A.~V.~Dossou-Olory.
	\newblock Cut and pendant vertices and the number of connected induced subgraphs of a graph.
	\newblock Preprint, arXiv:1910.04552, 2019 (Submitted).
	
	\bibitem{Entringer1976}
	R.~C.~Entringer, D.~E.~Jackson and D.~A.~Snyder. 
	\newblock Distance in graphs.
	\newblock {\em Czech. Math. J.}, 26(101): 283--296, 1976.
	
	\bibitem{knor2015mathematical}
    M.~Knor, R.~ {\v{S}}krekovski, and A. Tepeh.
	\newblock Mathematical aspects of Wiener index.
	\newblock {\em Ars Math. Contemporanea}, 11:327--352, 2016.
		
	\bibitem{LiWang}
	S.~Li and S.~Wang.
	\newblock Further analysis on the total number of subtrees of trees.
	\newblock {\em Electron. J. Combin.}, 19(4), \#P48, 2012.
	
	\bibitem{Maxwell2014}
	S.~Maxwell, M.~R.~Chance, and M.~Koyut\"{u}rk .
	\newblock Efficiently enumerating all connected induced subgraphs of a large molecular network.
	\newblock  In: A.~H.~Dediu,  C.~Mart\'{i}n-Vide, B.~Truthe (eds) {\em Algorithms for Computational Biology}. AlCoB 2014. Lecture Notes Comput. Sci., vol 8542. Springer, Cham.
	

	
	\bibitem{milo2002network}
	R.~Milo, S.~Shen-Orr, S.~Itzkovitz, N.~Kashtan, D.~Chklovskii, and U.~Alon.
	\newblock Network motifs: simple building blocks of complex networks.
	\newblock {\em Science}. 298(5594):824--827, 2002.
	
	
	\bibitem{szekely2005subtrees}
	L.~A.~Sz{\'e}kely and H.~Wang.
	\newblock On subtrees of trees.
	\newblock {\em Advances Appl. Math.}, 34(1):138--155, 2005.
	
	\bibitem{szekely2007binary}
	L.~A.~Sz{\'e}kely and H.~Wang.
	\newblock Binary trees with the largest number of subtrees.
	\newblock {\em Discr. Appl. Math.}, 155(3):374--385, 2007.
	
	
	\bibitem{tan2017wiener}
	S.-W.~Tan, Q.-L.~Wang and Y.~Lin.
	\newblock The Wiener index of unicyclic graphs given number of pendant vertices or cut vertices.
	\newblock {\em J. Appl. Math. Comput.}, 55(1--2):1--24, 2017.
	
	
	\bibitem{Uno2015}
	T.~Uno.
	\newblock Constant time enumeration by amortization.
	\newblock In: F.~Dehne, J.~R.~Sack, U.~Stege (eds) {\em Algorithms and Data Structures}. WADS 2015. Lecture Notes Comput. Sci., vol 9214. Springer, Cham.
	
	\bibitem{YanYeh2006}
	W.~Yan and Y.-N.~ Yeh.
	\newblock Enumeration of subtrees of trees.
	\newblock {\em Theor. Comput. Sci.}, 369:256--268, 2006.
	
	\bibitem{WagnerCorr}
	S.~Wagner.
	\newblock Correlation of graph-theoretical indices.
	\newblock {\em SIAM J. Discr. Math.}, 21(1): 33--46, 2007.
	
	
    \bibitem{Wiener1947}
    H.~Wiener. 
    \newblock Structural determination of paraffin boiling points. 
    \newblock {\em J. Amer. Chem. Soc.}, 69(1):17--20, 1947.
    	
\end{thebibliography}
\end{document}